\documentclass[12pt]{amsart}
\usepackage{amssymb,latexsym}
\usepackage{enumerate}
\usepackage{float}

\usepackage{mathtools}
\usepackage{amsfonts}
\usepackage{amssymb}
\usepackage{amsmath}
\usepackage{graphicx}
\usepackage{graphicx,color}

\makeatletter
\@namedef{subjclassname@2010}{%
  \textup{2010} Mathematics Subject Classification}
\makeatother

\begin{document}

\title[Fej\'er Polynomials and Control of Discrete Systems]{Fej\'er Polynomials and Control of Nonlinear Discrete Systems}
\author[Dmitrishin et al.]{D. Dmitrishin}
\address{Odessa National Polytechnic University, 1 Shevchenko Avenue, Odessa 65044, Ukraine}
\email{dmitrishin@opu.ua}
\author[]{P. Hagelstein}
\address{Department of Mathematics, Baylor University, Waco, Texas 76798}
\email{paul\!\hspace{.018in}\_\,hagelstein@baylor.edu}
\thanks{This work was partially supported by a grant from the Simons Foundation (\#208831 to Paul Hagelstein).}
\author[]{A. Khamitova}
\address{Georgia Southern University, Department of Mathematical Sciences,  203 Georgia Avenue,
 Statesboro, Georgia 30460-8093}
\email{anna\!\hspace{.018in}\_\,khamitova@georgiasouthern.edu}
\author[]{A. Korenovskyi}
\address{Odessa National University, Dvoryanskaya 2, Odessa 65000, Ukraine}
\email{anakor@paco.net}
\author[]{A. Stokolos}
\address{Georgia Southern University, Department of Mathematical Sciences,  203 Georgia Avenue,
 Statesboro, Georgia 30460-8093}
\email{astokolos@georgiasouthern.edu}

\date{}


\begin{abstract}

We consider optimization problems associated to a delayed feedback control (DFC) mechanism for stabilizing cycles of one dimensional discrete time systems.   In particular, we consider a delayed feedback control for stabilizing $T$-cycles of a differentiable function $f: \mathbb{R}\rightarrow\mathbb{R}$ of the form
$$x(k+1) = f(x(k)) + u(k)$$
where
$$u(k) = (a_1 - 1)f(x(k)) + a_2 f(x(k-T)) + \cdots + a_N f(x(k-(N-1)T))\;,$$
with $a_1 + \cdots + a_N = 1$.   Following an approach of Morg\"ul, we associate to each periodic orbit of $f$, $N \in \mathbb{N}$, and $a_1$, \ldots, $a_N$ an explicit  polynomial whose Schur stability corresponds to the stability of the DFC on that orbit.   We prove that, given any 1- or 2-cycle of $f$, there exist $N$ and $a_1$, $\ldots$, $a_N$ whose associated polynomial is Schur stable, and we find the minimal $N$ that guarantees this stabilization.    The techniques of proof will take advantage of extremal properties of the Fej\'er kernels found in classical harmonic analysis.
\end{abstract}

\maketitle

\newcommand{\ind}{1\hspace{-2.3mm}{1}}

\section{Introduction}
\newtheorem{thm}{Theorem}

 Problems related to the control of chaotic systems have received considerable attention in a number of disciplines, in particular engineering, physics, and mathematics.  In the foundational paper \cite{ogy90}, Ott, Grebogi, and Yorke observed that chaotic systems frequently contain unstable periodic orbits that may be stabilized by small time-dependent perturbations.   Specific control mechanisms for stabilizing chaotic systems were explored in subsequent papers such as  \cite{chen, shinbrot}.    A method of control of particular interest is the delayed feedback control (DFC) scheme introduced by Pyragas in \cite{pyragas92}.   The control in the Pyragas scheme is essentially a multiple of the difference between the current and one period delayed states of the system.  Distinctive  advantages of this scheme include the facts that the control term vanishes if the system is already in a periodic orbit and that the control term tends to zero as trajectories approach a given periodic orbit.  This DFC control mechanism finds many applications ranging from the stabilization of the modulation index of lasers to the suppression of pathological brain rhythms  \cite{bdg94, rp04}.

 In spite of its relative simplicity and broad range of application, the stability analysis of the DFC mechanism remains a delicate issue.  A particularly motivational paper to us in this regard is one of Morg\"ul.  In \cite{Morgul}, Morg\"ul considers the one-dimensional discrete time system
$$x(k+1) = f(x(k)) + u(k)\;,$$
with $k \in \mathbb{Z}$ being the time index and $f: \mathbb{R} \rightarrow \mathbb{R}$ an appropriately differentiable function.     We suppose that $f$ has a (possibly unstable) $T$-periodic orbit $\Sigma_T = \{x_0^\ast, \ldots, x_{T-1}^\ast\}$, where $f(x_{j\!\mod T}^\ast) = x_{j + 1\!\mod T}^\ast$\;.  A DFC control one may use to stabilize the orbit $\Sigma_T$ is$$u(k) = K(x(k) - x(k - T)).$$  As shown by Morg\"ul, we may analyze the local stability of this control by considering the auxiliary function $G: \mathbb{R}^{T+1} \rightarrow \mathbb{R}^{T+1}$ defined by
$G(z_1, \ldots, z_{T+1}) = (z_2, \ldots, z_{T+1}, f(z_{T+1}) + K(z_{T+1} - z_1))$.  We define $F: \mathbb{R}^{T+1} \rightarrow \mathbb{R}^{T+1}$ by $F = G^{T}$ (the composition of $G$ with itself $T$ times.)   Observe that $\Sigma_T':=\{x_0^\ast, \ldots, x^\ast_{T-1}, x_0^\ast\}$ is a fixed point of $F$.  The stability of $\Sigma_T$ under this control mechanism is equivalent to the stability at $\Sigma_T'$ of the system $\hat{x}(k+1) = F(\hat{x}(k))$.  The latter may be analyzed by finding the Jacobian of $F$ at $\Sigma_T'$.   The Jacobian of $F$ has a characteristic polynomial $p(\lambda)$, and a $T$-cycle $\Sigma_T$ is exponentially stable under the control provided that $p(\lambda)$ is Schur stable, i.e. all of its eigenvalues lie inside the unit disc of the complex plane.
 In \cite{Morgul}    Morg\"ul was able to explicitly provide the calculation of the above characteristic polynomial.

Motivated by this previous work, we are engaged in a research program involving a control that takes into account a deeper prehistory of the output values of a function.  In particular, we are considering a control of the form

\begin{equation}\tag{$\ast$}
u(k) = (a_1 - 1)f(x(k)) + a_2 f(x(k-T)) + \cdots + a_N f(x(k-(N-1)T))\;,\end{equation}
where $a_1 + \cdots + a_N = 1$, that takes into account not only the value of $f$ at $x_{k-T}$ but also $x_{k-2T}, x_{k-3T}, \ldots, x_{k-(N-1)T}$.  The reason for considering a control of this type is that, rather than having only one parameter $K$ that may be modified in our attempt to provide stability, we have a collection of parameters $a_1, \ldots, a_N$ at our disposal that may be adjusted to provide a more robust control mechanism.  In particular, limitations of the Pyragas control exhibited by Ushio \cite{Ushio} may under many conditions be bypassed by applying the above control for $N$ suitably large.  An explicit elementary example indicating the usefulness of this control is given in  the paper \cite{dhks2015}.
\\

Proceeding along the lines of the ideas of Morg\"ul, we associate to the above control a map $G:\mathbb{R}^{T(N-1) + 1} \rightarrow \mathbb{R}^{T(N-1)+1}$ defined by \\

$G(x_1, \ldots, x_{T(N-1) + 1}) =$\\ $(x_2, x_3, \ldots, x_{T(N-1) + 1}, a_1f(x_{T(N-1) + 1}) + a_2f(x_{T(N-2) + 1}) + \cdots + a_Nf(x_1))$.
\\

We  define $F = G^T$, and the stability of the above control on the cycle $\Sigma_T$ may be ascertained by the location of the roots of the characteristic polynomial of the Jacobian of $F$ at $(x_0^\ast, x_1^\ast, \ldots, x_{T(N-1)}^\ast)$ where for convenience we set $x_j^\ast = x^\ast_{{j\mod T}}\;.$  In the paper \cite{dhks2015}, we proved that this polynomial is given by
$$p(\lambda) =\lambda^{(N-1)T + 1} - \mu (q(\lambda))^{T}  \;,$$
where $\mu = f'(x_0^\ast) \cdots f'(x_{T-1}^\ast)$ and
$$q(\lambda) = a_1 \lambda^{N-1} + \cdots + a_{N-1}\lambda + a_N\;.$$

Having found this polynomial, we may naturally ask:  given a $T$-cycle of $f$ and an  associated multiplier $\mu < 1$ (we disregard multipliers $\mu \geq 1$ because any $T$-cycle of $f$ is automatically unstable under any control of the type we are considering), does there exist an $N$ and $a_1, \ldots, a_N$ satisfying $a_1 + \cdots + a_N = 1$ such that the $T$-cycle is stable under the control, e.g. that all of the roots of the polynomial $p(\lambda)$ above lie in the unit disc of $\mathbb{C}$?   Also, given $N$, 
what choice of $a_1$, $\ldots$, $a_N$ provides the largest open interval $\mathcal{I} \subset \mathbb{R}$
 such that $p(\lambda)$ is Schur stable whenever $\mu \in \mathcal{I}$?

In this paper, we address the above problems for the cases that $T = 1$ and $T = 2$.   In the $T=1$ case we prove the following:

\begin{thm}\label{thm1}
Let $T = 1$ and $N \in \mathbb{N}$.  Suppose  $\tilde{\mu} \in (-\cot^2\frac{\pi}{2(N+1)}, 1)$.   Then there exist  $a_1$, \ldots, $a_N$ satisfying $a_1 + \cdots + a_N = 1$ such that, if $\mu \in (\tilde{\mu}, 1)$, all the roots of the polynomial $p(\lambda)$ lie in the unit disc $\{z \in \mathbb{C} : |z| < 1\}$.
\end{thm}
 The optimality of the range of $\mu$ provided above is demonstrated by the following.

 \begin{thm}\label{thm2}
 Let $T=1$ and $N \in \mathbb{N}$.  If $\tilde{\mu} < - \cot^2\frac{\pi}{2(N+1)}$, then there is no choice of $a_1$, $\ldots$, $a_N$ satisfying $a_1 + \cdots + a_N = 1$ such that, for every $\mu \in (\tilde{\mu}, 1)$, all of the roots of the polynomial $p(\lambda)$ lie in the unit disc of $\mathbb{C}$.
 \end{thm}

  For the case that $T = 2$ we prove the following.

 \begin{thm}\label{thm3}
Let $T = 2$ and $N \in \mathbb{N}$.  Suppose $\tilde{\mu} \in (- N^2, 1)$. Then there exist $a_1$, \ldots, $a_N$ satisfying $a_1 + \cdots + a_N = 1$ such that, if $\mu \in (\tilde{\mu}, 1)$, all the roots of the polynomial $p(\lambda)$ lie in the unit disc $\{z \in \mathbb{C} : |z| < 1\}$.
\end{thm}
 The optimality of the range of $\mu$ provided above is demonstrated by the following.

 \begin{thm}\label{thm4}
 Let $T=2$ and $N \in \mathbb{N}$.  If $\tilde{\mu} < - N^2$ then there is no choice of $a_1$, $\ldots$, $a_N$ satisfying $a_1 + \cdots + a_N = 1$ such that, for every $\mu \in (\tilde{\mu}, 1)$, all of the roots of the polynomial $p(\lambda)$ lie in the unit disc of $\mathbb{C}$.
 \end{thm}

    These theorems indicate a very useful aspect of the type of control under consideration.  In particular, when $T=1$, if  $\tilde{\mu} \in (-\cot^2\frac{\pi}{2(N+1)}, 1)$, we may find $a_1, \ldots, a_N$ such that the control ($\ast$) stabilizes \emph{all} 1-cycles of a function $f$ so long as the associated multiplier $\mu$ lies \emph{anywhere} in the interval $(\tilde{\mu}, 1)$. A similar advantage holds for this type of control when 2-cycles are considered.  In this regard, our control avoids the deficiency of only being able to stabilize 1- and 2-cycles of orbits of $f$ associated to very particular multipliers.  We wish to thank J. P. \mbox{Kahane} for pointing out to us the desirability of constructing a control of this type.
    \\

We will see that these results are a consequence of the pioneering work \cite{fejer1915} of Fej\'er on nonnegative trigonometric polynomials.   The reader is quite likely familiar with the fact that the classical Fej\'er kernels in the theory of Fourier series are nonnegative.   Less well-known is the following: if
$$g (\theta) = 1 + \lambda_1 \cos \theta + \cdots + \lambda_n \cos n\theta$$  is nonnegative,
then
$$\left|\lambda_1\right| \leq 2 \cos\frac{\pi}{n+2}\;.$$   Moreover,
$$g(\theta) \leq n + 1\;.$$   Both of these inequalities are sharp.
 It will be shown that the bound on $\tilde{\mu}$ in Theorem 2 is equivalent to the first inequality above; the bound on $\tilde{\mu}$ for Theorem 4 is equivalent to the latter.  In that regard we see that classical inequalites associated to nonnegative trigonometric polynomials are closely related to problems of optimization in control theory.

\section{Preliminaries}

Fix $N \in \mathbb{N}$, $T\in \mathbb{N}$, and $a_1, \ldots, a_N$ such that $a_1 + \cdots + a_N = 1$, and let $p(\lambda)$ and $q(\lambda)$ be as in the previous section.   Observe that if $\mu = 0$, all the roots of $p(\lambda)$ are $0$ and accordingly lie in the unit disc $\mathbb{D}$ of the complex plane.   Now, as the roots of a polynomial vary continuously as a function of the coefficients of the polynomial \cite{harrismartin}, we see that all the roots of $p(\lambda)$ will lie in $\mathbb{D}$ provided that $\tilde{\mu} < \mu \leq 0$, where $\tilde{\mu}_T= \tilde{\mu}_T(T, N, a_1, \ldots, a_N)$ is defined by
$$\frac{1}{\tilde{\mu}_T} = \inf\{ z  < 0 : z = \frac{(q(\lambda))^T}{\lambda^{(N-1)T + 1}} \; \textup{for some} \; \lambda\; \textup{ in }\;\partial \mathbb{D}\}.  $$
Although elementary, these observations clarify considerably the desired coefficients $a_1, \ldots, a_N$ in the control ($\ast$): we seek $a_1, \ldots, a_N$ that will yield a polynomial $q(\lambda)$ such that the above infimum is as close to 0 as possible.

Noting that $\partial{\mathbb{D}} = \{e^{i \omega} : \omega \in [0,2\pi]\}$, we have that
\begin{align}\frac{1}{\tilde{\mu}_T} &= \inf_{\lambda \in \partial \mathbb{D}}\left\{\mathfrak{R}\left(\frac{(q(\lambda))^T}{\lambda^{(N-1)T + 1}}\right) : \mathfrak{I}\left(\frac{(q(\lambda))^T}{\lambda^{(N-1)T + 1}}\right) = 0\right\}\notag
\\&=\inf_{\lambda \in \partial \mathbb{D}}\left\{\mathfrak{R}\left(\frac{1}{\lambda}\left(a_1 + \cdots + a_N \lambda^{1 - N}\right)^{T}\right) : \mathfrak{I}\left(\frac{1}{\lambda}\left(a_1 + \cdots + a_N \lambda^{1 - N}\right)^{T}\right)= 0\right\}\notag
\end{align}


Recognizing that $\partial \mathbb{D} = \{e^{it} : t \in [0,2\pi)\}$, when $T=1$ we
have that

\begin{align}
\frac{1}{\tilde{\mu}_1} &= \inf_{t \in [0,2\pi)}\left\{\mathfrak{R}\left(\frac{1}{e^{it}}\left(a_1 + \cdots + a_N e^{i(1-N)t}\right)\right): \mathfrak{I}\left(\frac{1}{e^{it}}\left(a_1 + \cdots + a_N e^{i(1-N)t}\right)\right) = 0\right\}\notag
\\&= \inf_{t \in [0, \pi)} \left\{\sum_{j=1}^{N}a_j \cos jt : \sum_{j=1}^{N}a_j \sin jt = 0\right\}\;.\notag
\end{align}

It will later be important for us to know that $\tilde{\mu}_1$ is negative.   This is seen as follows.   Let $F(z) = \sum_{j=1}^N a_j z^j$.   Since $F(z)$ is nonzero and $F(0) = 0$, by the open mapping theorem we have $F(\partial \mathbb{D})$ intersects the negative $x$-axis.   Hence $\frac{1}{\tilde{\mu}_1} < 0$, and accordingly we have $\tilde{\mu}_1 < 0$.

When $T=2$, by recognizing that \mbox{$\{\lambda^2 : \lambda \in \partial \mathbb{D}\} = \partial \mathbb{D}$,} we have

\begin{align}
\frac{1}{\tilde{\mu}_2}&= \inf_{\lambda \in \partial\mathbb{D}}\left\{\mathfrak{R}\left(\frac{(q(\lambda^2))^2}{\lambda^{(2N-1)2 + 2}}\right) : \mathfrak{I}\left(\frac{(q(\lambda^2))^2}{\lambda^{2(N-1)2 + 2}}\right) = 0\right\}\notag
\\&=\inf_{t \in [0,2\pi)}\left\{\mathfrak{R}\left(\frac{1}{e^{i2t}}\left(a_1 + \cdots + a_N e^{i(1-N)2t}\right)^2\right): \mathfrak{I}\left(\frac{1}{e^{i2t}}\left(a_1 + \cdots + a_N e^{i(1-N)2t}\right)^2\right) = 0\right\}\;.\notag
\end{align}

Now, if $z \in \mathbb{C}$, we have that $\mathfrak{I}(z^2) = 0$ if and only if $\mathfrak{R}(z) = 0$ or that $z \in \mathbb{R}$.  Of course, if $z \in \mathbb{R}$ we have that $z^2 \geq 0$.  If we knew that $\tilde{\mu}_2$ were negative, we would have that the infimum would be associated only to values of $t \in [0, \pi)$ such that
$$\arg\left(\frac{1}{e^{it}}\left(a_1 + \cdots + a_N e^{i(1-N)2t}\right)\right) = \pm \frac{\pi}{2}\;.$$
For those values of $t$ we would have
$$ \mathfrak{I}\left(\frac{1}{e^{i2t}}\left(a_1 + \cdots + a_N e^{i(1-N)2t}\right)^2\right) = 0$$ if and only if $$\mathfrak{R}\left(\frac{1}{e^{it}}\left(a_1 + \cdots + a_N e^{i(1-N)2t}\right)\right) = 0.$$
If $\mathfrak{R}(z) = 0$, then $\mathfrak{R}(z^2) = - (\mathfrak{I}(z))^2$.    Accordingly, if $\tilde{\mu}_2$ were negative we would have

\begin{align}
\frac{1}{\tilde{\mu}_2}&= -\left(\inf_{t \in [0,2\pi)}\left\{  \mathfrak{I}\left( \frac{1}{e^{it}}\left(a_1 + \cdots + a_N e^{i(1-N)2t}\right)\right) : \mathfrak{R}\left(\frac{1}{e^{it}}\left(a_1 + \cdots + a_N e^{i(1-N)2t}\right)\right) = 0\right\} \right)^2 \notag
\\&= -\left(\inf_{t\in[0,2\pi)}\left\{\sum_{j=1}^N a_j \sin(2j-1)t : \sum_{j=1}^N a_j\cos(2j-1)t = 0\right\}\right)^2\;.\notag
\end{align}

      We now show that $\tilde{\mu}_2$ is indeed negative.   Define the function $F(z)$ by $F(z) = \sum_{j=1}^N a_j z^{2j-1}$.  $F$ is of course holomorphic and $F(0) = 0$.   By the open mapping theorem, since $F$ is nonzero we must have that $F(\partial \mathbb{D})$ intersects the imaginary axis away from the origin.   Let $t \in [0,2\pi)$ be such that $\mathfrak{I}(F(e^{it}) \neq 0$ and $\mathfrak{R}(F(e^{it}) = 0$.   As $\mathfrak{I}(F(e^{it})) = - \mathfrak{I}(F(e^{-it}))$ and $\mathfrak{R}(F(e^{it})) =  \mathfrak{R}(F(e^{-it}))$, we have that there exists $t \in [0,\pi)$ such that $\mathfrak{I}(F(e^{it}) \neq 0$ and $\mathfrak{R}(F(e^{it})) = 0$.   For that value of $t$, notice that $\sum_{j=1}^N a_j \sin(2j-1)t \neq 0$ and $\sum_{j=1}^{N} \cos(2j-1)t = 0$, and accordingly
      $$\mathfrak{I}\left( \frac{1}{e^{it}}\left(a_1 + \cdots + a_N e^{i(1-N)2t}\right)\right) \neq 0$$
      and
      $$\mathfrak{R}\left(\frac{1}{e^{it}}\left(a_1 + \cdots + a_N e^{i(1-N)2t}\right)\right) = 0\;.$$

      Now, if $z \in \mathbb{C}$ lies on the imaginary axes and away from the origin, we necessarily have that $z^2$ lies on the negative real axis.   Hence there exists $t \in [0,\pi)$ such that
      $$\mathfrak{R}\left(\frac{1}{e^{i2t}}\left(a_1 + \cdots + a_N e^{i(1-N)2t}\right)^2\right) < 0$$ and

      $$\mathfrak{I}\left(\frac{1}{\lambda}\left(a_1 + \cdots + a_N \lambda^{1 - N}\right)^{T}\right) = 0\;.$$  So $\tilde{\mu}_2$ is negative, as desired.

      Having computed $\tilde{\mu}_1$, we recognize that Theorems 1 and 2 are immediate consequences of the following:

      \begin{thm}\label{thm5}
   $$ \frac{-1}{\cot^2\frac{\pi}{2(N+1)}} = \sup_{a_1 + \cdots + a_N = 1} \inf_{t \in [0, \pi)} \left\{\sum_{j=1}^{N}a_j \cos jt : \sum_{j=1}^{N}a_j \sin jt = 0\right\}\;.
   $$
    Moreover, defining $a_1^0, \ldots, a_N^0$ by
    \begin{equation}
a_{j}^0 = 2\cdot \tan\frac{\pi }{2(N+1)} \cdot (1-\frac{j}{N+1} )\cdot \sin \frac{\pi j}{N+1} ,\; j=1,\, \ldots \, ,N, \notag
\end{equation}
and setting $a_1^\epsilon = \frac{a_1^0 + \epsilon}{1 + \epsilon}$, $a_j^\epsilon = \frac{a_j^0}{1 + \epsilon} ,\; j=2, \ldots, N$\;, we have
\mbox{ $a_1^\epsilon + \cdots + a_N^\epsilon = 1$} and
 \begin{equation}
 \frac{-1}{\cot^2\frac{\pi}{2(N+1)}} = \lim_{\epsilon \rightarrow 0^+}  \inf_{t \in [0, \pi)} \left\{\sum_{j=1}^{N}a_j^\epsilon \cos jt : \sum_{j=1}^{N}a_j^\epsilon \sin jt = 0\right\}\;.\notag
 \end{equation}
   \end{thm}

  Having computed $\tilde{\mu}_2$, we recognize that Theorems 3 and 4 are immediate consequences of

  \begin{thm}\label{thm6}
  $$\frac{-1}{N^2} = \sup_{a_1 + \cdots + a_N = 1}  -\left(\inf_{t\in[0,2\pi)}\left\{\sum_{j=1}^N a_j \sin(2j-1)t : \sum_{j=1}^N a_j\cos(2j-1)t = 0\right\}\right)^2\;.
  $$   Moreover, defining $a_1^0, \ldots, a_N^0$ by
  \begin{equation}
a_j^0 =\frac{2(N-j)+1}{N^{2} } ,\; j=1,\, \ldots \, ,N,\notag
\end{equation}
and setting  $a_1^\epsilon = \frac{a_1^0 + \epsilon}{1 + \epsilon}$, $a_j^\epsilon = \frac{a_j^0}{1 + \epsilon} ,\; j=2, \ldots, N$\;, we have
\mbox{ $a_1^\epsilon + \cdots + a_N^\epsilon = 1$} and
$$\frac{-1}{N^2} = \lim_{\epsilon \rightarrow 0+}  -\left(\inf_{t\in[0,2\pi)}\left\{\sum_{j=1}^N a_j^\epsilon \sin(2j-1)t : \sum_{j=1}^N a_j^\epsilon\cos(2j-1)t = 0\right\}\right)^2\;.
  $$
 \end{thm}

\section{Factorization of Conjugate Trigonometric Polynomials}

The following lemma, in many respects a real analogue of Bezout's theorem, will in  subsequent sections be of considerable use to us in proving Theorems 5 and 6.

\newtheorem{lem}{Lemma}
\begin{lem}

Let $$C(t) = \sum_{j=1}^n a_j \cos jt\;,\;\;\;\;S(t) = \sum_{j=1}^n a_j \sin jt$$
be a pair of conjugate trigonometric polynomials with real coefficients.   Moreover, suppose that
$$ S(t_1) = \cdots = S(t_m) = 0\;,\;\;\;\;\;C(t_1) = \cdots = C(t_m) = \gamma\;,$$
where $t_1\;, \ldots, t_m$ lie in the interval $(0,\pi)$ and $2m \leq n$.  Then the trigonometric polynomials $C(t)$ and $S(t)$ admit the presentation
$$
C(t) = \gamma + \prod_{j=1}^{m}(\cos t - \cos t_j) \sum_{k=m}^{n-m}\alpha_k \cos kt\;\;,\;\; S(t) = \prod_{j=1}^m(\cos t - \cos t_j)\sum_{k=m}^{n-m}\alpha_k \sin kt\;,$$
where $\alpha_m = -2^m \gamma$ and the coefficients $\alpha_m , \ldots, \alpha_{n-m}$ can be uniquely expressed in terms of $\gamma , a_1, \ldots, a_n$.
\end{lem}

\begin{proof}
Consider the algebraic polynomial
$$F(z) = -\gamma + \sum_{j=1}^n \alpha_j z^j\;.$$
Note that
$$C(t)= \gamma + \mathfrak{R}\left\{F(e^{it})\right\}\;\textup{,}\;S(t) = \mathfrak{I}\left\{F(e^{it})\right\}\;.$$

By the Fundamental Theorem of Algebra, $F(z)$ has $n$ roots, and $2m$ of these are provided by $e^{it_1}, \ldots, e^{it_j}$ and $e^{-it_1}, \ldots, e^{-it_j}$. Accordingly, there exist numbers $\beta_1\;\textup{,}\ldots\textup{,}\beta_{n-2m}$ such that
$$F(z) = \left(\prod_{j=1}^{m}\left(z - e^{it_j}\right)\left(z - e^{-it_j}\right)\right)\left(-\gamma + \sum_{k=1}^{n - 2m}\beta_k z^k\right)\;.$$

Observe that
\begin{align}F(z) = \prod_{j=1}^m\left(z - e^{it_j}\right)\left(z - e^{-it_j}\right) &= \prod_{j=1}^m\left(z^2 - 2z \cos t_j + 1\right)\notag \\& = 2^m z^m\prod_{j=1}^m\left(\frac{1}{2}\left(z + \frac{1}{z}\right) - \cos t_j\right)\;.\notag \end{align}  Hence
$$F(z) = \prod_{j=1}^m\left(\frac{1}{2}\left(z + \frac{1}{z}\right) - \cos t_j\right)\left(-2^m\gamma z^m + 2^m \sum_{k=1}^{n-2m}\beta_k z^{m+k}\right)\;.$$
Hence
$$\mathfrak{R}\left\{F(e^{it})\right\} = \prod_{j=1}^m\left(\cos t - \cos t_j\right)\left(-2^m\gamma \cos mt + 2^m \sum_{k = m+1}^{n-m}\beta_{k-m}\cos kt\right)\;\textup{,}$$
$$\mathfrak{I}\left\{F(e^{it})\right\} = \prod_{j=1}^m\left(\cos t - \cos t_j\right)\left(-2^m\gamma \sin mt + 2^m \sum_{k = m+1}^{n-m}\beta_{k-m}\sin kt\right)\;\textup{,}$$
implying the desired result with $\alpha_m = -2^m\gamma$ and $\alpha_{m+k} = 2^m \beta_k$ for $k=1, \ldots, n-2m$.
\end{proof}

\section{Nonlocal Separation from Zero}
 Recall that the polynomimal $q(\lambda)$ is defined by $q(\lambda) = a_1 \lambda^{N-1} + \cdots + a_{N-1}\lambda + a_N$.   For technical reasons that will arise in the next section we will need a uniform lower bound of the distance between $q(\partial \mathbb{D})$ and the origin that holds for all $a_1 , \ldots, a_N$ such that $a_1 + \cdots a_N = 1$.   This lower bound is provided by the following.

 \begin{lem}
 Let $F(z) = a_1 z + \cdots + a_n z^n$, where $a_j \in \mathbb{C}$ for each $j$.   Then the set $F(\mathbb{D})$ contains a disc centered at the origin with radius $\frac{1}{2^n}\sum_{j=1}^n |a_j|$.
 \end{lem}

\begin{proof}
Let $\gamma \notin F(\mathbb{D})$.   Observe that if $z \in \mathbb{D}$, $|F(z)| \leq \sum_{j=1}^n|a_j|$, so such a value of $\gamma$ indeed does exist.    Note  that the polynomial $F(z) - \gamma$ does not have a root inside of $\mathbb{D}$.  Hence all of the roots of the polynomial $z^n(F(\frac{1}{z}) - \gamma)$  lie within $\bar{\mathbb{D}}$.  Now, $z^n(F(\frac{1}{z}) - \gamma)$ may be expressed as
\begin{align}
z^n(F(\frac{1}{z}) - \gamma) &= -\gamma z^n + a_1 z^{n-1} + \cdots + a_n \notag
\\&= -\gamma\left(z^n - \frac{a_1}{\gamma}z^{n-1} - \cdots - \frac{a_n}{\gamma}\right)\;.\notag
\end{align}
Applying Vieta's theorem to the polynomial in parentheses, we obtain the estimate
$$\left|\frac{a_j}{\gamma}\right| \leq  \binom{n}{j} \; j=1, \ldots, n$$
that in turn implies $\sum_{j=1}^n \left|\frac{a_j}{\gamma}\right| \leq 2^n -1$ and hence
$|\gamma| \geq \frac{1}{2^n - 1}\sum_{j=1}^n |a_j|$.
\end{proof}

\begin{lem}
Let $C(t) = \sum_{j=1}^n a_j \cos jt$ and $S(t) = \sum_{j=1}^n a_j \sin jt$ be conjugate trigonometric polynomials, where $\sum_{j=1}^n a_j = 1$.  Then
$$\min\{C(t) : S(t) = 0\} \leq -\frac{1}{2^n}\;. $$
\end{lem}
\begin{proof}
This follows immediately from the previous lemma, considering the values of the associated $F$ on $\partial \mathbb{D}$.
\end{proof}
\section{The $T = 1$ case}
In this section we prove Theorem 5 and obtain Theorems 1 and 2 as corollaries.  The strategy involves first noting that $$ \sup_{a_1 + \cdots + a_N = 1} \inf \left\{\sum_{j=1}^{N}a_j \cos jt : \sum_{j=1}^{N}a_j \sin jt = 0\right\}$$
is bounded above by
$$\sup_{a_1 + \cdots + a_N = 1} \inf \left\{\sum_{j=1}^N a_j \cos jt : t = \pi \;\textup{ or }\sum_{j = 1}^N a_j \sin jt\;\;\textup{ changes sign at }t\right\}$$
and then showing that this latter supremum equals $-\tan^2 \frac{\pi}{2(N+2)}$.  We will do this by seeing that the polynomial $S(t)$ generating the desired supremum is of the form $S(t) = \sin(t) P(t)$, where $P(t)$ is a nonnegative trigonometric polynomial associated to an optimization problem related to an inequality due to Fej\'er.  We conclude the proof by showing that, defining $a_1^0, \ldots, a_N^0$ by
    \begin{equation}
a_{j}^0 = 2\cdot \tan\frac{\pi }{2(N+1)} \cdot (1-\frac{j}{N+1} )\cdot \sin \frac{\pi j}{N+1} ,\; j=1,\, \ldots \, ,N,
\notag \end{equation}
and setting $a_1^\epsilon = \frac{a_1^0 + \epsilon}{1 + \epsilon}$, $a_j^\epsilon = \frac{a_j^0}{1 + \epsilon} ,\; j=2, \ldots, N$\;, we have
\mbox{ $a_1^\epsilon + \cdots + a_N^\epsilon = 1$} and
 \begin{equation}
 \frac{-1}{\cot^2\frac{\pi}{2(N+1)}} = \lim_{\epsilon \rightarrow 0^+}  \inf_{t \in [0, \pi)} \left\{\sum_{j=1}^{N}a_j^\epsilon \cos jt : \sum_{j=1}^{N}a_j^\epsilon \sin jt = 0\right\}\;.\notag \end{equation}
\\

Given $a_1$, \ldots $a_N$ such that $a_1 + \cdots + a_N = 1$, define the associated pair of conjugate trigonometric polynomials
$$C(t) = \sum_{j=1}^N a_j \cos jt\;,\;\;\;S(t) = \sum_{j=1}^N a_j \sin jt\;.$$
 The function $\rho_1 (a_1, \ldots, a_N)$ is given by
$$\rho_1(a_1, \ldots, a_N) = \min\left\{C(t) : t \in \mathcal{T} \cup \{\pi\}\right\}\;,$$
where $\mathcal{T}$ is the set of points in $(0,\pi)$ where $S(t)$ changes sign.

\begin{lem}\label{l4}
There exists a conjugate pair of trigonometric polynomials $\left(C^0(t) , S^0(t)\right)$, the sum of the  coefficients of either polynomial being 1,  such that
$$\sup_{(a_{1}, \ldots, a_{N}) : \atop a_1 + \cdots + a_N = 1}\{\rho_1(a_1, \ldots, a_N)\} = \min\{C^0(t) : t \in \mathcal{T}^0\cup\{\pi\}\}\;,$$ where $\mathcal{T}^0$ is the set of points in  $(0,\pi)$ such that the function $S^0(t)$ changes sign.

\end{lem}
\begin{proof}
We define the set $A_R \subset \mathbb{R}^n$ by
$$A_R = \left\{(a_1, \ldots, a_N) : \sum_{j=1}^N a_j = 1\;, \sum_{j=1}^N |a_j| \leq R\right\}\;.$$   If $R \geq 1$, the function $\rho_1 (a_1, \ldots, a_N)$ is upper semi-continuous on $A_R$.  $A_R$ being a compact set, by the Weierstrass maximum theorem (see, e.g., \cite{borden}) we have that the supremum of $\rho_1$ acting on $A_R$ is achieved on $A_R$.

It remains to show that the this supremum is independent of $R$ for $R$ sufficiently large.  Considering the case that $a_1 = 1$, $a_2 = \cdots = a_N = 0$, we immediately realize that this supremum is greater than or equal to -1.   By Lemma 2, we also realize that this supremum could not be realized by $(a_1, \ldots, a_N)$ such that $\sum_{j=1}^N |a_j| > 2^N$.   Accordingly, we realize the supremum is achieved for some $(a_1, \ldots, a_N) \in A_R$ for any value of $R$ exceeding $2^N$.

\end{proof}

 We shall call a pair of conjugate trigonometric polynomials $(C(t), S(t))$ \emph{optimal} if the pair satisfies the hypotheses of the above lemma.

\begin{lem}\label{l5}
If the polynomial $S(t) = \sum_{j=1}^N a_j \sin jt$ has a sign change in $(0,\pi)$, the associated pair of conjugate polynomials $(C(t), S(t))$ cannot be optimal.
\end{lem}
\begin{proof}

Let $(C^0(t), S^0(t))$ be an optimal pair of conjugate polynomials.  We proceed by contradiction.   Suppose that $S^0(t)$ had a sign change in $(0, \pi)$.   Then $\mathcal{T} = \left\{t_1, \ldots, t_q\right\}$ would be a nonempty set consisting of all the zeros of $S^0(t)$ on $(0,\pi)$ where $S^0(t)$ changes sign.    Note that since $S^0(t)$ is a polynomial of degree $n$, we have that $S^0(t)$ has less than or equal to $2n$ roots on $(-\pi, \pi)$ (see, e.g., \cite{powell}), and hence by symmetry considerations we have $q \leq n-1$.  We assume without loss of generality that
$$C^0(t_1) = C^0(t_2) = \cdots = C^0(t_m) < C^0(t_{m+1}) \leq C^0({t_{m+2}}) \leq \cdots \leq C^0({t_q})\;.$$
Now, observe that the polynomial
$$-C^0(t_1) + C(t) + i S(t)$$
has at most $n$ roots on $(-\pi, \pi]$ and hence by symmetry considerations has no more than $n/2$ roots  on $(0,\pi)$.   Hence we have  $2m \leq n$.   So we may apply Lemma 1, yielding the factorizations of $S^0(t)$, $C^0(t)$
$$
C^0(t) = C^0(t_1) + \prod_{j=1}^{m}(\cos t - \cos t_j) \sum_{k=m}^{n-m}\alpha_k \cos kt\;,
$$
$$ S^0(t) = \prod_{j=1}^m(\cos t - \cos t_j)\sum_{k=m}^{n-m}\alpha_k \sin kt\;,$$
where the coefficents $\alpha_m$, \ldots, $\alpha_{n-m}$ are expressed uniquely in terms of $C^0(t_1)$ and the coefficents of $S^0(t)$ and $C^0(t)$, moreover having that $\alpha_m = -2^m C^0(t_1)$.

Note that $\sum_{k=m}^{n-m} \alpha_k \sin kt$ may have zeros other than $t_1, \ldots, t_m$ on the interval $(0,\pi)$, but there are only finitely many of them and on none of them does $C^0$ vanish.
\\

 We now define the functions $S(\theta_1, \ldots, \theta_m; t)$, $C(\theta_1, \ldots, \theta_m;t)$ by
 $$
 S(\theta_1, \ldots, \theta_m; t) = N(\theta_1, \ldots, \theta_m) \cdot  \prod_{j=1}^m(\cos t - \cos \theta_j)\sum_{k=m}^{n-m}\alpha_k \sin kt\;,$$
 $$
 C(\theta_1, \ldots, \theta_m; t) = N(\theta_1, \ldots, \theta_m) \cdot \left(-\frac{\alpha_m}{2^m} + \prod_{j=1}^{m}(\cos t - \cos \theta_j) \sum_{k=m}^{n-m}\alpha_k \cos kt\right)\;,
$$
 where the normalizing factor $N(\theta_1, \ldots, \theta_m)$ is such that, when $S(\theta_1, \ldots, \theta_m; t)$ is expressed in a form of a sine series $\sum_{j=1}^N \tilde{\alpha}_j \sin jt$, we have \mbox{$\sum_{j=1}^N \tilde{\alpha}_j = 1$.}  Observe that we know $S(\theta_1, \ldots, \theta_m; t)$ \emph{can} be expressed in terms of a sine series of this form because of the elementary trigonometric identity
 $$\sin a \cos b = \frac{1}{2}\left[\sin(a+b) + \sin(a-b)\right]\;$$ and the fact that we already know that $S^0(t_1, \ldots, t_m; t)$ may be expressed as the sine series $S^0(t)$.     Also observe that the trigonometric identity
 $$\cos a \cos b = \frac{1}{2}\left[\cos(a+b) + \cos(a-b)\right]\;$$ and the fact that $(C^0(t), S^0(t))$ form a conjugate pair together imply that
 $$C(\theta_1, \ldots, \theta_m; t) = \sum_{j=1}^N \tilde{\alpha}_j \cos jt\;,$$
 telling us that $C(\theta_1, \ldots, \theta_m; t)$ and $S(\theta_1, \ldots, \theta_m;t)$ may be expressed as conjugate trigonometric polynomials in $t$, each having a sum of associated coefficients being 1.

 Note that since $C(\theta_1, \ldots, \theta_m; 0) = 1$, we have
 $$N(\theta_1, \ldots, \theta_m) = \frac{1}{-\frac{\alpha_m}{2^m} + \prod_{j=1}^m (1 - \cos \theta_j)\sum_{k=m}^{n-m}\alpha_k}\;.$$
 This leads to the expressions
 $$
 C(\theta_1, \ldots, \theta_m; t) = \frac{-\frac{\alpha_m}{2^m} + \prod_{j=1}^m(\cos t - \cos \theta_j)\sum_{k=m}^{n-m}\alpha_k \cos kt}{-\frac{\alpha_m}{2^m} + \prod_{j=1}^m(1 - \cos \theta_j)\sum_{k=m}^{n-m}\alpha_k}\;,$$

 $$
 S(\theta_1, \ldots, \theta_m; t) = \frac{\prod_{j=1}^m(\cos t - \cos \theta_j)\sum_{k=m}^{n-m}\alpha_k \sin kt}{-\frac{\alpha_m}{2^m} + \prod_{j=1}^m(1 - \cos \theta_j)\sum_{k=m}^{n-m}\alpha_k}\;.$$

 We recognize that $S(t_1, \ldots, t_m;t)$ changes sign on $(0,\pi)$ exactly at $t_1, \ldots, t_q$, and accordingly $S(\theta_1, \ldots, \theta_m; t)$ changes sign on $(0,\pi)$ exactly at $\theta_1, \ldots, \theta_m, t_{m+1}, \ldots, t_q$, provided that $(\theta_1, \ldots, \theta_m)$ is sufficiently close to $(t_1, \ldots, t_m)$ in $\mathbb{R}^m$.

 The idea now is to show that for some minor perturbation $(\theta_1, \ldots, \theta_m)$ of $(t_1, \ldots, t_m)$, we have \emph{both} $$\min\{C(\theta_1, \ldots, \theta_m; t) : S(\theta_1, \ldots, \theta_m;t) \;\textup{changes sign at}\;t\;, t \in (0, \pi)\}$$ \emph{and} $$C(\theta_1, \ldots, \theta_m; \pi)$$ are larger than  $\min\{C^0(t_1), C^0(\pi)\}$, contradicting the fact that $(C^0(t), S^0(t))$ is an optimal pair.  Note that since $C^0(t_1) < C^0(t_j)$ for $j > m$, there exists $\epsilon > 0$ so that, provided $|t_j - \theta_j| < \epsilon$ for $j=1, \ldots, m$, we will have

$$\min\{C(\theta_1, \ldots, \theta_m; t) : S(\theta_1, \ldots, \theta_m;t) \;\textup{changes sign at}\;t\;, t \in (0,\pi)\} = $$ $$\min\{C(\theta_1, \ldots, \theta_m; \theta_j) : j\in \{1, \ldots, m\}\}\;.$$
\\

 The proof now involves the consideration of three cases: either \newline \mbox{$C^0(t_1) < C^0(\pi)$,} $C^0(\pi) \leq C^0(t_1) \leq 1$, or $C^0(t_1) > 1$.
 \\

\noindent \emph{Case 1: $C^0(t_1) < C^0(\pi)$:}
\\

As proven in Section 2 using the open mapping theorem, we have $C^0(t_1) = -\frac{\alpha_m}{2^m} < 0$.   Hence $\alpha_m > 0$.  Since $C^0(0) = 1$, we have
$$-\frac{\alpha_m}{2^m} + \prod_{j=1}^{m}(1 - \cos t_j) \sum_{k=m}^{n-m}\alpha_k = 1\;,
$$
and hence
$$
 \sum_{k=m}^{n-m}\alpha_k = \frac{1 + \frac{\alpha_m}{2^m}}{\prod_{j=1}^{m}(1 - \cos t_j)} > 0\;.
 $$

Note

 $$
 C(\theta_1, t_2, t_3, \ldots, t_m; \theta_1) = \frac{-\frac{\alpha_m}{2^m}}{-\frac{\alpha_m}{2^m} + (1 - \cos \theta_1)  \prod_{j=2}^m(1 - \cos t_j)\sum_{k=m}^{n-m}\alpha_k}\;.$$

We have shown that $\alpha_m > 0$ and that $\sum_{k=m}^{n-m}\alpha_k > 0$.  So this formula implies that $$C(t_1 + \epsilon, t_2, t_3, \ldots, t_m; t_1 + \epsilon) - C(t_1, t_2, t_3, \ldots, t_m; t_1)\;,$$ viewed as a function of $\epsilon$, changes sign at  $\epsilon = 0$.  As \mbox{$S(t_1 + \epsilon , t_2, t_3, \ldots,t_m; t)$} changes sign at $t_1 + \epsilon$ for all $\epsilon$ sufficiently close to 0, we see that $(C^0(t), S^0(t))$ cannot be an optimal pair.
\\


\noindent \emph{Case 2: $C^0(\pi) \leq C^0(t_1) \leq 1$:}
\\

We first observe that
\begin{align}
C^0(\pi) &= -\frac{\alpha_m}{2^m} + \prod_{j=1}^{m}(-1 - \cos t_j) \sum_{k=m}^{n-m}(-1)^k\alpha_k \;\notag
\\&= -\frac{\alpha_m}{2^m} + (-1)^m\prod_{j=1}^m(1 + \cos t_j) \sum_{k=m}^{n-m}(-1)^k \alpha_k\;,\notag
\end{align}

\begin{align}
C(\theta_1, \ldots, \theta_m; \pi) &=  \frac{-\frac{\alpha_m}{2^m} + \prod_{j=1}^m(-1 - \cos \theta_j)\sum_{k=m}^{n-m}(-1)^k\alpha_k}{-\frac{\alpha_m}{2^m} + \prod_{j=1}^m(1 - \cos \theta_j)\sum_{k=m}^{n-m}\alpha_k}\;,\notag
\\&= \frac{-\frac{\alpha_m}{2^m} + (-1)^m\prod_{j=1}^m(1 + \cos \theta_j)\sum_{k=m}^{n-m}(-1)^k\alpha_k}{-\frac{\alpha_m}{2^m} + \prod_{j=1}^m(1 - \cos \theta_j)\sum_{k=m}^{n-m}\alpha_k}\;,\notag
\end{align}
and
$$C(t_1, \ldots, t_m; \pi) = C^0(\pi)\;.$$

Since we are in the case that $C^0(\pi) \leq C^0(t_1) = -\frac{\alpha_m}{2^m}$, we have $(-1)^m\sum_{k=m}^{n-m}(-1)^k\alpha_k \leq 0$, and hence
$$(-1)^m\prod_{j=1}^m(1 + \cos \theta_j)\sum_{k=m}^{n-m}(-1)^k\alpha_k$$ is nondecreasing with respect to each of the parameters $\theta_1$, \ldots, $\theta_m$.  As in Case 1, we also have that $C^0(0) = 1$ implies
$$-\frac{\alpha_m}{2^m} + \prod_{j=1}^{m}(1 - \cos t_j) \sum_{k=m}^{n-m}\alpha_k = 1\;,
$$
and hence
$$
 \sum_{k=m}^{n-m}\alpha_k = \frac{1 + \frac{\alpha_m}{2^m}}{\prod_{j=1}^{m}(1 - \cos t_j)} \geq 0\;
 $$
since $-\frac{\alpha_m}{2^m} = C^0(t_1) \leq 1$ by hypothesis.

Note that if $\theta_j$ is very close to $t_j$ for $j=1, \ldots, m$, $C(\theta_1, \ldots, \theta_m;\pi)$ is very close to $C^0(\pi)$ and hence the denominator in the expression for $C(\theta_1, \ldots, \theta_m;\pi)$ above is very close to 1.   As the quotient is negative when the $\theta_j$ are close to the respective $t_j$, one of the sums $\sum_{k=m}^{n-m}(-1)^k \alpha_k$, $\sum_{k=m}^{n-m}\alpha_k$ must be nonzero.  As, at $\theta_1 = t_1$,  the numerator is negative and nondecreasing in $\theta_1$ and the denominator is positive and nondecreasing in $\theta_1$ and at least one of the numerator or denominator is strictly increasing, we must have that $C(\theta_1, t_2, \ldots, t_m; \pi) > C^0(\pi)$ when $0 < \theta_1 - t_1 < \epsilon$ for some $\epsilon > 0$.   (We remark that, in general, the fact that a quotient has an increasing numerator and denominator does \emph{not} imply that the quotient is increasing; here it is essential to recognize that at $t_1$ the numerator is negative and the denominator is positive.) 

Note that the above argument dispatches with the case that $C^0(\pi) < C^0(t_1)$.  Suppose now $C^0(\pi) = C^0(t_1)$.   As by the above argument we already know that $C(\theta_1, t_2, \ldots, t_m; \pi)$ is increasing in $\theta_1$ for $\theta_1$ near $t_1$,  it  suffices to show that $C(\theta_1, t_2, \ldots, t_m; \theta_1)$ and $C(\theta_1, t_2, \ldots, t_m; t_j)$ are increasing in $\theta_1$  for $j=2, 3, \ldots, m$.  Note we have that
\begin{align}C(\theta_1, t_2, \ldots, t_m; \theta_1) &= C(\theta_1, t_2, \ldots, t_m; t_j)\notag \\&= \frac{-\frac{\alpha_m}{2^m}}{\frac{-\alpha_m}{2^m} + (1 - \cos \theta_1)\prod_{j=2}^m (1 - \cos t_j)\sum_{k=m}^{n-m}\alpha_k}\;.\notag
  \end{align}
  Observe that $\sum_{k=m}^{n-m}\alpha_k$ cannot be 0, as otherwise the quotient would be identically 1, contradicting the fact that it tends to a negative number as $\theta_1$ tends to $t_1$.   As we have already observed in this case that $\sum_{k=m}^{n-m}\alpha_k \geq 0$, we conclude that $\sum_{k=m}^{n-m}\alpha_k > 0$.  Hence the denominator in the quotient above is strictly increasing in $\theta_1$ for $\theta_1$ near $t_1$.  Similarly to the argument above, as the denominator is near 1 for $\theta_1$ near $t_1$ and the numerator is negative (here $C^0(\pi) = C^0(t_1)$ and hence both are negative, and note $- \frac{\alpha_m}{2^m}$ equals both of these in this scenario), the expressions above for $C(\theta_1, t_2, \ldots, t_m; \theta_1)$ and $C(\theta_1, t_2, \ldots, t_m; t_j)$  are increasing in $\theta_1$ for $j=2, 3, \ldots, m$ when $\theta_1$ is near $t_1$.  This contradicts that $(C^0(t), S^0(t))$ forms an optimal pair.
\\

\noindent \emph{Case 3: $C^0(t_1) > 1$:}
\\

Applying Lemma 1, we have that the members of the optimal pair $(C^0(t), S^0(t))$ may be expressed as
$$C^0(t) = -\frac{\beta_1}{2} + (\cos t - \cos t_1)\sum_{j=1}^{n-1}\beta_j \cos jt\;,$$
$$S^0(t) = (\cos t - \cos t_1)\sum_{j=1}^{n-1}\beta_j \sin jt\;,$$ where
$$-\frac{\beta_1}{2} + (1 - \cos t_1)\sum_{j=1}^{n-1}\beta_j = 1 \;, \; C^0(t_1) = -\frac{\beta_1}{2} > 1\;,$$
and $$
C^0(\pi) = - \frac{\beta_1}{2}  (1 + \cos t_1)\sum_{j=1}^{n-1}(-1)^j\beta_j < 0\;.$$  (Recall that one of the $C^0(t_j)$ or $C^0(\pi)$ must be negative by Lemma 3.)  These yield
$$\sum_{j=1}^{n-1}\beta_j = \frac{\frac{\beta_1}{2} + 1}{1 - \cos t_1} < 0\;,\;\sum_{j=1}^{n-1}(-1)^j\beta_j < 0\;.$$
Consider now the collection of pairs of  polynomials $(C(\theta, t), S(\theta, t))$ conjugate in the variable $t$ defined by
$$C(\theta, t) = \frac{-\frac{\beta_1}{2} + (\cos t - \cos \theta)\sum_{j=m}^{n-1}\beta_j \cos jt}{-\frac{\beta_1}{2} + (1 - \cos \theta)\sum_{j=1}^{n-1}\beta_j}\;,$$
$$S(\theta, t) = \frac{(\cos t - \cos \theta)\sum_{j=m}^{n-1}\beta_j \sin jt}{-\frac{\beta_1}{2} + (1 - \cos \theta)\sum_{j=1}^{n-1}\beta_j}\;.$$
Analogous to the previous two cases, we have $C(t_1, t) = C^0(t)$ and $S(t_1, t) = S^0(t)$. Note that
$$C(\theta, \pi) = \frac{-\frac{\beta_1}{2} - (1 + \cos \theta)\sum_{j=1}^{n-1} (-1)^j \beta_j}{-\frac{\beta_1}{2} + (1 - \cos \theta)\sum_{j=1}^{n-1}\beta_j}\;.$$
Using the quotient rule, we see that the function $C(\theta, \pi)$ is either monotonic on $(0,\pi)$ as a function of $\cos \theta$ or is identically constant.   In the first case, there exists $\theta_1$ close to $t_1$ such that $C(\theta_1, \pi) > C(t_1,\pi) = C^0(\pi)$, contradicting that $(C^0(t), S^0(t))$ is an optimal pair.

Hence we may assume without loss of generality that $C(\theta, \pi) \equiv \gamma$. Since $C(t_1, \pi) = C^0(\pi) < 0$ (either $C^0(t_1)$ or $C^0(\pi)$ must be less than zero and the former does not hold in Case 3) we must have $\gamma < 0$.  Hence
$$
\lim_{\theta \rightarrow \pi} C(\theta, \pi) = \frac{-\frac{\beta_1}{2}}{-\frac{\beta_1}{2} + 2 \sum_{j=1}^{n-1}\beta_j} < 0\;.$$  Since $-\frac{\beta_1}{2} > 0$, we must have that $\sum_{j=1}^{n-1}\beta_j < 0$ and that the function $-\frac{\beta_1}{2} + (1 - \cos \theta)\sum_{j=1}^{n-1}\beta_j$ is positive at 0 and negative for $\theta$ close to $\pi$.   Hence by the intermediate value theorem we realize that there exists $\theta_2$ so that
$$C(\theta_2, \theta_2) = \frac{-\frac{\beta_1}{2}}{-\frac{\beta_1}{2} + (1 - \cos \theta_2)\sum_{j=1}^{n-1}\beta_j} < 2^n \gamma\;.$$

Now, the absolute value of a trigonometric polynomial does not exceed the sum of the absolute value of the coefficients, and hence the sum of the absolute values of the coefficients of the polynomial $C(\theta_2, t)$ must exceed $-2^n \gamma$.   But then by Lemma 2 we must have $|C(\theta_2, \pi)| > |\gamma|$ contradicting that $C(\theta_2, \pi) = \gamma$.   Hence the pair $(C^0(t), S^0(t))$ cannot be optimal.
\end{proof}

We now make some observations that will motivate our proof of Theorem 5.  We have shown that if $(C(t), S(t))$ is a pair of optimal polynomials, then $S(t)$ cannot have a sign change in $(0,\pi)$.  This leads directly to considerations of nonnegative trigonometric polynomials.   An early estimate associated to the coefficients of nonnegative trigonometric polynomials is due to Fej\'er in \cite{fejer1915}; in particular he proved that if the trigonometric polynomial
$$1 + \lambda_1 \cos t + \cdots + \lambda_n \cos nt$$  is nonnegative,
then
$$\left|\lambda_1\right| \leq 2 \cos\frac{\pi}{n+2}\;.$$   An explicit example of a polynomial satisfying the upper bound for $\lambda_1$ was given by Egerv\'ary and Sz\'asz, who in \cite{evszasz} proved that

\begin{equation}\left|\sum_{k=0}^{n}\sin\frac{(k+1)\pi}{n+2}e^{ikt}\right|^2 = \frac{n+2}{2} + \sum_{k=1}^n \left\{(n - k + 1)\cos\frac{k\pi}{n+2} + \frac{\sin\frac{(k+1)\pi}{n+2}}{\sin\frac{\pi}{n+2}}\right\}\cos kt\;.\notag
\end{equation}

Using appropriate substitution, by defining the coefficients $b_k$ by $b_0 = 1$,

$$b_k = \frac{(N-k+2)\sin((k+1)\frac{\pi}{N+1}) - (N-k)\sin((k-1)\frac{\pi}{N+1})}{(N+1)\sin \frac{\pi}{N+1}}\;,\; k=2, \ldots, N-1$$
we have
\begin{equation}\sum_{k=0}^{N-1}b_k \cos(kt) = \frac{2}{N+1}\left|\sum_{k=0}^{N-1}\sin((k+1)\frac{\pi}{N+1})e^{ikt}\right|^2\;\label{szasz}\end{equation}
and hence the above Fej\'er polynomial is nonnegative.



\begin{proof}[Proof of Theorems \ref{thm1}, \ref{thm2}, and \ref{thm5}]

     As indicated previously, Theorems 1 and 2 follow from Theorem 5, so it suffices to prove the latter.

Let $(C^0(t), S^0(t))$ be an optimal pair of conjugate trigonometric polynomials, where $S(t) = \sum_{j=1}^N a_j^0 \sin jt$.   Via the trigonometric identity
$$\sin a \cos b = \frac{1}{2}[\sin(a+b) + \sin(a-b)]$$
we express $S^0(t)$ as
$$S^0(t) = \sin t \cdot \left(\gamma_1^0 + 2 \gamma_2^0 \cos t + \cdots + 2 \gamma_N^0 \cos (N-1)t\right)\;,$$
where there is a bijective correspondence between $a_1^0, \ldots, a_N^0$ and $\gamma_1^0, \ldots, \gamma_N^0$.   Since $a_1^0 + \cdots + a_N^0 = 1$, one can use the above identity to show that $\gamma_1^0 + \gamma_2^0 = 1$.  Note moreover that we have $$-a_1^0 + a_2^0 - a_3^0 + \cdots + (-1)^Na_N^0 = -\gamma_1^0 + \gamma_2^0.$$   Since by the previous lemma we have that $S^0(t)$ does not change sign in $(0,\pi)$, by Lemma 4 we have that
\begin{align}
\sup_{(a_{1}, \ldots, a_{N}) : \atop a_1 + \cdots + a_N = 1}\{\rho_1(a_1, \ldots, a_N)\} \notag &= C^0(\pi) \notag\\
&=-a_1^0 + a_2^0 - a_3^0 + \cdots + (-1)^N a_N^0\notag\\&= \gamma^0_1 + \gamma^0_2\;. \notag
\end{align}

Since $S^0(t)$ has no sign change in $(0,\pi)$, neither does $S^0(t)/ \sin t$.   Hence by the Fej\'er inequality for non-negative polynomials \cite{fejer1915}, we have
$$|\gamma^0_2| \leq \cos \frac{\pi}{N+1}\cdot |\gamma^0_1|\;.$$
Accordingly we have
$$C^0(\pi) \leq \max_{\gamma_1, \gamma_2}\left\{-\gamma_1 + \gamma_2 : \gamma_1 + \gamma_2 = 1\;\textup{and} |\gamma_2| \leq \cos \frac{\pi}{N+1}\cdot |\gamma_1|\right\}\;.$$
One can compute that this maximum is achieved by
$$\gamma_1 = \frac{1}{1 + \cos \frac{\pi}{N+1}}, \gamma_2 = \frac{\cos \frac{\pi}{N+1}}{1 + \cos \frac{\pi}{N+1}}$$
and equals
$$-\frac{1 - \cos\frac{\pi}{N+1}}{1 + \cos \frac{\pi}{N+1}} = -\tan^2\frac{\pi}{2(N+1)}\;.$$  In particular, we then have that
$$C^0(\pi) \leq -\tan^2\frac{\pi}{2(N+1)}\;.$$

This upper bound on $C^0(\pi)$ may be realized by using the polynomial featured in Equation \ref{szasz} above, appropriately scaled.  If we define the coefficients $b_k$ by $b_0 = 1$,

$$b_k = \frac{(N-k+2)\sin((k+1)\frac{\pi}{N+1}) - (N-k)\sin((k-1)\frac{\pi}{N+1})}{(N+1)\sin \frac{\pi}{N+1}}\;,\; k=2, \ldots, N-1$$
we have
$$\sum_{k=0}^{N-1}b_k \cos(kt) = \frac{2}{N+1}\left|\sum_{k=0}^{N-1}\sin((k+1)\frac{\pi}{N+1})e^{ikt}\right|^2\;$$
is nonnegative for $t \in (0,\pi)$.  Setting $$\gamma_1 = \frac{1}{1 + \cos\frac{\pi}{N+1}}$$ and $$\gamma_k = \frac{1}{2(1 + \cos\frac{\pi}{N+1})}b_{k-1},$$ we indeed have that
$$\gamma_1 + 2 \gamma_2 \cos t + \cdots + 2\gamma_N \cos(N-1)t$$ is nonnegative on $(0, \pi)$ and
\begin{align}\gamma_2 = \frac{1}{2(1 + \cos\frac{\pi}{N+1})}b_{1} &= \frac{1}{2(1 + \cos\frac{\pi}{N+1})}\frac{(N+1)\sin\frac{2\pi}{N+1}}{(N+1)\sin\frac{\pi}{N+1}}\notag\\&=\frac{1}{2(1 + \cos\frac{\pi}{N+1})}\frac{2\sin\frac{\pi}{N+1}\cos\frac{\pi}{N+1}}{\sin\frac{\pi}{N+1}}\notag\\&=\frac{\cos\frac{\pi}{N+1}}{1 + \cos\frac{\pi}{N+1}}.\notag\end{align}
Moreover, setting
$$S(t) = \sin(t) \left(\gamma_1 + 2\gamma_2\cos t + \cdots + 2 \gamma_N \cos(N-1)t\right)\;,$$
we have
$$S(t) = \sum_{j=1}^N a_j \sin jt\;$$  is a nonnegative trigonometric polynomial on $(0,\pi)$, where here
\begin{equation}\label{aj}a_j = \gamma_j - \gamma_{j+2} = 2 \tan\frac{\pi}{2(N+1)}\cdot\left(1 - \frac{j}{N+1}\right)\cdot \sin\frac{\pi j}{N+1}\;, j=1, \ldots, N\;,\notag \end{equation}
 setting $\gamma_{N+1} = \gamma_{N+2} = 0$ for convenience.  Note that
$$\sum_{j=1}^N a_j = \gamma_1 + \gamma_2 = 1\;.$$   Hence, defining $C(t) = \sum_{j=1}^N a_j \cos jt$, we see that $(C(t), S(t))$ forms an optimal conjugate pair of trigonometric polynomials and that $$C(\pi) = -\tan^2\frac{\pi}{2(N+1)}\;.$$   Lemma 4 then immediately implies
$$\sup_{(a_{1}, \ldots, a_{N}) : \atop a_1 + \cdots + a_N = 1}\{\rho_1(a_1, \ldots, a_N)\} = -\tan^2\frac{\pi}{2(N+1)}\;.$$   In particular, we have now that

$$\sup_{a_1 + \cdots + a_N = 1} \inf \left\{\sum_{j=1}^N a_j \cos jt : t = \pi \;\textup{ or }\sum_{j = 1}^N a_j \sin jt\;\;\textup{ changes sign at }t\right\} $$ $$= -\tan^2\frac{\pi}{2(N+1)}\;.$$

It remains to show

$$ \sup_{a_1 + \cdots + a_N = 1} \inf \left\{\sum_{j=1}^{N}a_j \cos jt : \sum_{j=1}^{N}a_j \sin jt = 0\right\} = -\tan^2\frac{\pi}{2(N+1)}\;.$$   Note there is content in this last step as, although we know the Fej\'er polynomial above is nonnegative, we do not have precise information as to where it vanishes.

We proceed as follows.   Define the coefficients $a_j$ as above in (\ref{aj}).  For $\epsilon > 0$, define the conjugate pair of trigonometric polynomials $(C^\epsilon(t), S^\epsilon(t))$ by

 $$C^\epsilon(t) = \frac{a_1 + \epsilon}{1 + \epsilon} \cos t + \frac{a_2}{1 + \epsilon}\cos 2t + \cdots + \frac{a_N}{1 + \epsilon}\cos Nt\;;$$

$$S^\epsilon(t) = \frac{a_1 + \epsilon}{1 + \epsilon} \sin t + \frac{a_2}{1 + \epsilon}\sin 2t + \cdots + \frac{a_N}{1 + \epsilon}\sin Nt\;.$$

 Since $\sum_{j=1}^N a_j = 1$, we have $$\frac{a_1 + \epsilon}{1 + \epsilon} + \frac{a_2}{1 + \epsilon} + \cdots + \frac{a_N}{1 + \epsilon} = 1\;.$$   Moreover,

 $$S^\epsilon(t) = \frac{S(t)}{1 + \epsilon} + \frac{\epsilon}{1 + \epsilon}\sin t\;,$$ and hence $S^\epsilon(t) > 0$ for $t \in (0,\pi)$.   Note $$C^\epsilon (\pi) = \frac{1}{1 + \epsilon}\left(-\tan^2\frac{\pi}{2(N+1)}\right) + \frac{\epsilon}{1 + \epsilon}$$ and hence
 $$\lim_{\epsilon \rightarrow 0^+} C^\epsilon(\pi) = -\tan^2\frac{\pi}{2(N+1)}\;.$$   The desired result then holds.

\end{proof}

\section{The $T=2$ case}

In this section we prove Theorem 6 and obtain Theorems 3 and 4 as corollaries.   Our strategy is similar to the one employed in our proof of Theorem 5.  In particular, we note that
$$ \sup_{a_1 + \cdots + a_N = 1} -\left(\inf_{t \in [0, 2\pi)}\left\{\sum_{j=1}^N a_j \sin(2j-1)t : \sum_{j=1}^N a_j\cos(2j-1)t = 0\right\}\right)^2 $$ is bounded from above by

$$- \inf_{a_1 + \cdots + a_N = 1}\left(\sup_{t \in [0, 2\pi)} \left\{\sum_{j=1}^N a_j \sin(2j-1)t : t = \pm \frac{\pi}{2}\;\;\textup{or}\;\; \sum_{j=1}^N a_j\cos(2j-1)t  \;\;\textup{changes sign at}\;t\right\}\right)^2 $$
and subsequently show that this latter expression equals
$-\frac{1}{N^2}$.   Afterwards, we shall again employ Fej\'er polynomials to show that,  defining $a_1^0, \ldots, a_N^0$ by
  \begin{equation}
a_j^0 =\frac{2(N-j)+1}{N^{2} } ,\; j=1,\, \ldots \, ,N, \notag
\end{equation}
and setting  $a_1^\epsilon = \frac{a_1^0 + \epsilon}{1 + \epsilon}$, $a_j^\epsilon = \frac{a_j^0}{1 + \epsilon} ,\; j=2, \ldots, N$\;, we have
\mbox{ $a_1^\epsilon + \cdots + a_N^\epsilon = 1$} and
$$\frac{-1}{N^2} = \lim_{\epsilon \rightarrow 0+}  -\left(\inf_{t \in [0,2\pi)} \left\{\sum_{j=1}^N a_j^\epsilon \sin(2j-1)t : \sum_{j=1}^N a_j^\epsilon\cos(2j-1)t = 0\right\}\right)^2\;.
  $$
\\

Given $a_1, \ldots, a_N$ such that $a_1 + \cdots + a_N = 1$, we define the associated pair $(C(t), S(t))$ of conjugate trigonometric polynomials by
$$C(t) = \sum_{j=1}^N a_j \cos(2j-1)t\;\;,\;\;S(t) = \sum_{j=1}^N a_j \sin(2j-1)t\;.$$
The function $\rho_2(a_1, \ldots, a_N)$ is given by
$$\rho_2(a_1, \ldots, a_N) = \sup \left\{S(t) : t \in \mathcal{T} \cup \{\pm\frac{\pi}{2}\}\right\}\;,$$
where $\mathcal{T}$ is the set of points $t$ in $(-\frac{\pi}{2}, 0) \cup (0,\frac{\pi}{2})$ where $C(t)$ changes sign and $S(t)$ is positive.

\begin{lem}\label{l6}
There exists a pair of conjugate trigonometric polynomials $(C^0(t), S^0(t))$ such that
$$\inf_{(a_1, \ldots, a_N): \atop a_1 + \cdots + a_N = 1} \{\rho_2(a_1, \ldots, a_N)\} = \sup \left\{S^0(t) : t \in \mathcal{T}^0 \cup \left\{\pm\frac{\pi}{2}\right\}\right\}\;,$$
where $\mathcal{T}^0$ is the set of points in $(-\frac{\pi}{2},\frac{\pi}{2})$ where $S^0(t)$ is positive and $C^0(t)$ changes sign.
\end{lem}
\begin{proof}
The proof is virtually identical to that of Lemma \ref{l4}\;.
\end{proof}

If a conjugate pair $(C(t), S(t))$ satisfies the condition of Lemma 6, we will refer to it as being \emph{optimal}.

\begin{lem}\label{l7}
If the polynomial $C(t) = \sum_{j=1}^N a_j \cos(2j-1)t$ has a sign change in $(-\frac{\pi}{2},\frac{\pi}{2})$, the associated pair of conjugate polynomials $(C(t), S(t))$ cannot be optimal.
\end{lem}
\begin{proof}
We proceed by contradiction.   Suppose $(C^0(t), S^0(t))$ were an optimal pair but that $C^0(t)$ had a sign change in $(-\frac{\pi}{2},\frac{\pi}{2})$.  Then there would exist a nonempty set $\mathcal{T} = \{t_1, \ldots, t_q\}$ consisting of the points in $(-\frac{\pi}{2},\frac{\pi}{2})$ where $C^0(t)$ changes sign.   $C^0(t)$ being a polynomial of degree $N$, we have that $q \leq 2N-1$. We assume without loss of generality that $S^0(\frac{\pi}{2}) \geq 0$ and that
$$S^0(t_1) = S^0(t_2) = \cdots = S^0(t_m) > S^0(t_{m+1}) \geq S^0(t_{m+2}) \geq \cdots \geq S^0(t_q)\;.$$

As in the proof of Lemma \ref{l5}, the proof now involves considering three cases: either $S^0(t_1) > S^0 (\frac{\pi}{2})$, $S^0(\frac{\pi}{2}) \geq S^0(t_1) \geq 0$, or $S^0(t_1) < 0$.
\\

\noindent \emph{Case 1: $S^0(t_1) > S^0 (\frac{\pi}{2})$:}
\\

Observe that the polynomials $C^0(t)$, $S^0(t)$ may be reexpressed as
$$C^0(t) = \frac{1}{2\sin t}\sum_{j=1}^N \hat{a}_j \sin 2jt\;,$$
$$S^0(t) = \frac{1}{2 \sin t}\left(\sum_{j=1}^N \hat{a}_j - \sum_{j=1}^N \hat{a}_j \cos 2jt \right)\;,$$
where $\hat{a}_j = a_j - a_{j+1}$, $j=1, \ldots, N$, setting here $a_{N+1} = 0 $ for convenience.   As $2 \sin t \cdot C^0(t)$ and $\sum_{j=1}^N \hat{a}_j - 2\sin t \cdot S^0(t)$ are conjugate trigonometric polynomials, the first of which being a sine polynomial vanishing on $\{t_1, \ldots, t_q\}$, Lemma 1 implies that
$$C^0(t) = \frac{1}{2 \sin t} \cdot \left(\cos 2t - \cos 2t_1\right)\sum_{j=1}^{N-1}\alpha_j \sin 2jt\;,$$
$$S^0(t) = \frac{1}{2 \sin t}\left(a_1 + \frac{\alpha_1}{2} - \left(\cos 2t - \cos 2t_1\right)\sum_{j=1}^{N-1}\alpha_j \cos 2jt\right)\;,$$
where $\alpha_j$, $j=1, \ldots, N-1$ are uniquely determined by $a_1, \ldots, a_N$ and $S^0(t_1)$, with
$\frac{\alpha_1}{2} = 2 S^0(t_1) \sin t_1 - a_1$. Note that as $C^0(0) = 1$, implying $\lim_{t \rightarrow 0}C^0(t) = 1$, we have
$$(1 - \cos 2t_1)\sum_{j=1}^{N-1}j \alpha_j = 1\;,$$ yielding that $\sum_{j=1}^{N-1}j \alpha_j > 0$.

We now construct the auxiliary trigonometric polynomials $S(\theta,t)$ and $C^0(\theta,t)$, defined by
$$C(\theta, t) = N(\theta)\frac{1}{2\sin t}(\cos 2t - \cos 2\theta)\sum_{j=1}^{N-1}\alpha_j \sin 2jt\;,$$
$$S(\theta, t) = N(\theta)\frac{1}{2\sin t}\left(a_1 + \frac{\alpha_1}{2} - (\cos 2t - \cos 2\theta)\sum_{j=1}^{N-1}\alpha_j \cos 2jt\right)\;,$$
where the normalization factor $N(\theta)$ is such that the sum of the coefficents of each of the polynomials $C(\theta, t)$ and $S(\theta, t)$ (expressed respectively as cosine and sine polynomials) is 1.   Note that the set of sign changes  in $(0, \pi/2)$  for $C(\theta, t)$ is $\mathcal{T}_\theta = \{\theta, t_2, \ldots, t_q\}.$   Observe also that $C(t_1, t) = C^0(t)$ and $S(t_1, t) = S^0(t)$.   We may find the normalizing factor $N(\theta)$ by noting that $C(\theta, 0) = 1$ implies
$$N(\theta) = \frac{1}{(1 - \cos 2\theta)\sum_{j=1}^{N-1} j \alpha_j}\;.$$   Hence the polynomials $C(\theta, t)$ and $S(\theta, t)$ may be expressed as
$$C(\theta, t) = \frac{1}{(1 - \cos 2\theta)\sum_{j=1}^{N-1}j \alpha_j}\cdot\frac{1}{2\sin t}(\cos 2t - \cos 2\theta)\sum_{j=1}^{N-1}\alpha_j \sin 2jt\;,$$
$$S(\theta, t) = \frac{a_1 + \frac{\alpha_1}{2} - (\cos 2t - \cos 2\theta)\sum_{j=1}^{N-1}\alpha_j \cos 2jt}{\left((1 - \cos 2\theta)\sum_{j=1}^{N-1}j \alpha_j\right)2\sin t }\;.$$

We now show that for some value of $\theta$ the value of $\rho_2(a_1, \ldots, a_n)$ associated to the pair $C(\theta, t), S(\theta,t)$ is less than that of the pair $(C^0(t), S^0(t))$, implying that the pair $(C^0(t), S^0(t))$ is not optimal.

Note that $$S(\theta, \theta) =\frac{a_1 + \frac{\alpha_1}{2}}{2 \sin \theta (1 - \cos 2\theta)(\sum_{j=1}^{N-1}j \alpha_j)}$$
and
$$S(\theta, t_k) = \frac{a_1 + \frac{\alpha_1}{2} - (\cos 2t_k - \cos 2\theta)\sum_{j=1}^{N-1}\alpha_j \cos 2jt_k}{2 \sin t_k (1 - \cos 2\theta)(\sum_{j=1}^{N-1}j \alpha_j)}\;; k = 2, \ldots, m\;.$$
Now, since $S^0(t_1) = S^0(t_k)$, $k = 2, \ldots, m$, we have
$$\sum_{j=1}^{N-1}\alpha_j \cos 2jt_k = \frac{a_1 + \frac{\alpha_1}{2}}{\sin t_1} \cdot \frac{\sin t_1 - \sin t_k}{\cos 2t_k - \cos 2t_1} > 0\;.$$  Hence the functions $S(\theta, \theta)$ and $S(\theta, t_k)$, $k=2, \ldots, m$ are all decreasing with respect to $\theta$ for $\theta \in (0, \pi / 2)$.   Accordingly, as $C(\theta, t)$ and $S(\theta, t)$ are continuous in $\theta$ and $t$, for sufficiently small $\epsilon > 0$ we have that $0 < \theta - t_1 < \epsilon$ implies
$$\max \left\{S(\theta, \theta), S(\theta, t_2), \ldots, S(\theta, t_q), S(\theta, \frac{\pi}{2})\right\}$$ is strictly less than
$$\max \left\{S^0(t_1), \ldots, S^0(t_q), S^0(\frac{\pi}{2})\right\}\;,$$
implying that $(C^0(t), S^0(t))$ is not an optimal pair.
\\

\noindent \emph{Case 2: $S^0(\frac{\pi}{2}) > S^0 (t_1) \geq 0$:}
\\

Observe that

$$S^0\left(\frac{\pi}{2}\right) = \frac{1}{2}\left(a_1 + \frac{\alpha_1}{2} + \left(1 + \cos 2t_1\right)\sum_{j=1}^{N-1}\left(-1\right)^j\alpha_j\right)$$ and
 $$ S\left(\theta, \frac{\pi}{2}\right) = \frac{a_1 + \frac{\alpha_1}{2} + \left(1 + \cos 2\theta\right)\sum_{j=1}^{N-1}(-1)^j\alpha_j}{2(1 - \cos 2\theta)\sum_{j=1}^{N-1}j \alpha_j}\;.$$

We have in this case that  $S^0(\frac{\pi}{2}) \geq \frac{1}{2 \sin t_1}(a_1 + \frac{\alpha_1}{2}) \geq 0$ and hence $\sum_{j=1}^{N-1}(-1)^j\alpha_j \geq 0$.  Arguing as in Case 1 we have that $S(\theta, \frac{\pi}{2}), S(\theta, \theta)$, and $S(\theta, t_k)$, $k=2, \ldots, m$ are decreasing in $\theta$, contradicting the optimality of $(C^0(t), S^0(t))$.
\\

\noindent \emph{Case 3: $S^0(t_1) < 0$:}
\\

Note that since in this case $S^0(t_1) < 0$ we have $a+1 + \frac{\alpha_1}{2} < 0$.   As one of $S^0(\frac{\pi}{2})$, $S(t_1)$ must exceed 0, we have $S^0(\frac{\pi}{2}) > 0$, from which the above formula for $S^0(\frac{\pi}{2})$ implies $\sum_{j=1}^{N-1}(-1)^j \alpha_j > 0$.   Hence the function $S(\theta, \theta)$ is decreasing in $\theta$ and the pair $(C^0(t), S^0(t))$  cannot be optimal.
\end{proof}

\begin{proof}[Proof of Theorems \ref{thm3}, \ref{thm4}, and  \ref{thm6}]

As indicated previously, it suffices to prove Theorem \ref{thm6}, as Theorems 3 and 4 follow.

By Lemma \ref{l7}, we have that
$$\inf_{(a_1, \ldots, a_N): \atop a_1 + \cdots + a_N = 1} \{\rho_2(a_1, \ldots, a_N)\} = \min_{a_1 + \cdots + a_N = 1}\left\{\left|S\left(\frac{\pi}{2}\right)\right| : C(t) > 0, t \in (0,\frac{\pi}{2})\right\}\;.$$

Note that the cosine polynomial $C(t) = \sum_{j=1}^N a_j \cos(2j-1)t$ may be written as
$$C(t) = \cos t \cdot \left(\gamma_1 + 2 \gamma_2 \cos 2t + \cdots + 2 \gamma_N \cos 2(N-1)t\right)\;,$$ where
$\gamma_1 + 2 \sum_{j=2}^N\gamma_j = \sum_{j=1}^{N} a_j = 1$.  Via the identity $\cos a \cos b = \frac{1}{2}[\cos(a + b) + \cos(a - b)]$ we have that
$$S\left(\frac{\pi}{2}\right) = a_1 - a_2 + \cdots + (-1)^N a_N = \gamma_1\;.
$$
Let now
$$\rho_2 = \min_{a_1 + \cdots + a_N = 1}\left\{\left|S\left(\frac{\pi}{2}\right)\right| : C(t) \geq 0, t \in \left(0, \frac{\pi}{2}\right)\right\}\;.$$
Then
$$\rho_2 = \min\left\{|\gamma_1| : \gamma_1 + 2 \sum_{j=1}^{N}\gamma_j = 1, C(t) \geq 0 \textup{ for } t \in (0,\frac{\pi}{2})\right\}\;.$$

Now, a nonnegative trigonometic polynomial of the form
$$g(\theta) = 1 + \lambda_1 \cos \theta + \mu_1 \sin \theta + \lambda_2 \cos 2\theta + \mu_2 \sin 2\theta + \cdots + \lambda_n \cos n \theta + \mu_n \sin n \theta$$
satisfies the inequality $$0 \leq g(\theta) \leq n + 1$$
 (see, e.g, \cite{fejer1915} or Problem 50 of \cite{polyaszego}.)   Accordingly, if $(C^0(t), S^0(t))$ is an optimal pair where $\gamma_1^0$ is associated to $C^0(t)$ as $\gamma_1$ is to $C(t)$ above, $\frac{C^0(0)}{\gamma_1^0} \leq N$ and hence $\rho_2 \leq \frac{1}{N}$.

With the above estimate, we have now established that

$$ \sup_{a_1 + \cdots + a_N = 1} -\left(\inf\left\{\sum_{j=1}^N a_j \sin(2j-1)t : \sum_{j=1}^N a_j\cos(2j-1)t = 0\right\}\right)^2 $$ is bounded from above by $-\frac{1}{N^2}$.  It remains to show that the supremum actually equals $-\frac{1}{N^2}$.  To that end, consider the optimal pair of trigonometric polynomials $(C^0(t), S^0(t))$ associated to the classical Fej\'er kernel, where
$$C^0(t) = \left(\frac{\sin Nt}{N \sin t}\right)^2 \cos t = \cos t \left(\frac{1}{N} + 2 \sum_{j=2}^N \frac{N - j + 1}{N^2}\cos 2(j-1)t\right)\;.$$
Notice here that the associated $\gamma_1^0$ is the (optimal) $\frac{1}{N}$ and the associated $\gamma_j^0$, $j=2, \ldots, N$ are given by
$$\gamma_j^0 = \frac{N - j + 1}{N^2}.$$
Setting for convention $\gamma_{N+1}^0 = 0$, the associated $a_j^0$ satisfy \mbox{$a_j^0 = \gamma_j^0 + \gamma_{j+1}^0$,} and hence $$a_j^0 = \frac{2(N - j) + 1}{N^2}; j=1, \ldots, N\;.$$

We now define the pair of conjugate polynomials $(C^\epsilon(t), S^\epsilon(t))$ by

$$C^\epsilon(t) = \sum_{j=1}^N a_j^\epsilon \cos(2j-1)t\;,S^\epsilon(t) = \sum_{j=1}^N a_j^\epsilon \sin(2j-1)t\;, $$ where
$$a_1^\epsilon = \frac{a_1^0 + \epsilon}{1 + \epsilon}\;,\;a_j^\epsilon = \frac{a_j^0}{1 + \epsilon}\;, j=2, \ldots, N.$$
Note that we indeed have $\sum_{j=1}^N a_j^\epsilon = 1$ and moreover that
$$C^\epsilon(t) = \frac{C^0(t)}{1 + \epsilon} + \frac{\epsilon}{1 + \epsilon}\cos t\;.$$  So $C^\epsilon(t) > 0$ for all $t \in (0, \frac{\pi}{2})$ and $\epsilon > 0$.  Hence
$$ \sup_{a_1 + \cdots + a_N = 1} -\left(\inf\left\{\sum_{j=1}^N a_j \sin(2j-1)t : \sum_{j=1}^N a_j\cos(2j-1)t = 0\right\}\right)^2  \geq - \left(S^\epsilon(\frac{\pi}{2})\right)^2\;.$$   As $$\lim_{\epsilon \rightarrow 0^+} S^\epsilon(\frac{\pi}{2}) = \frac{1}{N}\;,$$ the desired result holds.\end{proof}

\section{Future Directions}

The agenda for future work on this aspect of control theory is clear: if $\mu < 1$ and $T$ is an integer larger than 2, may we find $a_1, \ldots, a_N$ satisfying $a_1 + \cdots + a_N = 1$ such that all of the roots of $p(\lambda)$ lie in the unit disc, where
$$p(\lambda) =\lambda^{(N-1)T + 1} - \mu (q(\lambda))^{T}  \;,$$
with
$$q(\lambda) = a_1 \lambda^{N-1} + \cdots + a_{N-1}\lambda + a_N\;?$$
 If so, what is the infimum of the values of $\mu$ for which this can be done?
The reader may be somewhat surprised that we have been able to resolve this question for the $T = 1,2$ cases but the cases for higher values of $T$ remain.  In this regard, we should note that in the $T = 1, 2$ cases we were able to take advantage of some basic facts in complex analysis (such as if the square of a complex number $z$ is negative, then $z$ lies on the imaginary axis) that enabled us to reduce the problem to issues regarding nonnegative trigonometric polyomials on real line.   Such reductions are unavailable to us when $T \geq 3$ and we find ourselves in a position of needing to articulate and resolve issues associated to at present admittedly vague notions of ``complex-valued Fej\'er polynomials'' and a complex analytic analogue of the Fej\'er-Riesz theorem.  This is a subject of ongoing research.

\end{document}